\documentclass[aap,preprint]{imsart}

\usepackage{amsthm,amsmath,amssymb, verbatim, comment, color, amsfonts,amsxtra}
\RequirePackage[colorlinks,citecolor=blue,urlcolor=blue]{hyperref}
\RequirePackage[OT1]{fontenc}

\startlocaldefs

\newtheorem{theorem}{Theorem}

\newtheorem{corollary}[theorem]{Corollary}
\newtheorem{definition}[theorem]{Definition}
\newtheorem{lemma}[theorem]{Lemma}
\newtheorem{proposition}[theorem]{Proposition}

\theoremstyle{remark}
\newtheorem{remark}[theorem]{Remark}

 \renewcommand{\phi}{\varphi}
 
 \newcommand{\OB}{Ob{\l}{\'o}j} 
\newcommand{\HL}{Henry-Labord{\`e}re}

\newcommand{\E}{\mathbb{E}}
\renewcommand{\P}{\mathbb{P}}

\newcommand{\R}{\mathbb{R}}

\newcommand{\be}{\begin{equation}}
\newcommand{\ee}{\end{equation}}

\DeclareMathOperator{\supp}{supp}

\numberwithin{equation}{section}
\numberwithin{theorem}{section}

\usepackage{enumerate}

\setattribute{journal}{name}{}

\endlocaldefs

\begin{document}

\begin{frontmatter}

\title{Optimal Brownian Stopping between radially symmetric marginals in general dimensions}
\runtitle{Optimal Skorokhod Embedding between radially symmetric marginals}


\thankstext{}{The two first-named authors are partially supported by  the 
Natural Sciences and Engineering Research Council of Canada (NSERC). Y. H. Kim is also supported by an Alfred P. Sloan Research Fellowship. T. Lim was partly supported by a doctoral graduate fellowship from the University of British Columbia, by the Austrian Science Foundation (FWF) through grant Y782, and by the European Research Council under the European Union's Seventh Framework Programme (FP7/2007-2013) / ERC grant agreement no. 335421. Part of this research was done while the authors  were visiting the Fields institute in Toronto during the thematic program on ``Calculus of Variations'' in Fall 2014. We are thankful for the hospitality and the great research environment that the institute provided. 
}

\begin{aug}
\author{\fnms{Nassif} \snm{Ghoussoub}\thanksref{m1}
\ead[label=e1]{nassif@math.ubc.ca}}
\author{\fnms{Young-Heon} \snm{Kim}\thanksref{m1}
\ead[label=e2]{yhkim@math.ubc.ca,}}
\and
\author{\fnms{Tongseok} \snm{Lim}\thanksref{m2}
\ead[label=e3]{Tongseok.Lim@maths.ox.ac.uk}}

\runauthor{N. Ghoussoub, Y-H Kim and T. Lim}

\affiliation{The University of British Columbia\thanksmark{m1} and the University of Oxford\thanksmark{m2}}

\address{Department of Mathematics\\University of British Columbia\\ Vancouver, V6T 1Z2 Canada\\
\printead{e1}\\
\phantom{E-mail:\ }\printead*{e2}}

\address{Mathematical Institute\\University of Oxford\\
Woodstock Road, Oxford, OX2 6GG UK\\
\printead{e3}}
\end{aug}





\begin{abstract}
Given an initial (resp., terminal) probability measure $\mu$ (resp., $\nu$) on $\R^d$, we characterize those optimal stopping times $\tau$ that  maximize  or minimize the functional $\E |B_0 - B_\tau|^{\alpha}$, $\alpha > 0$, where $(B_t)_t$ is Brownian motion with initial law $B_0\sim \mu$ and with final distribution --once stopped at $\tau$-- equal to $B_\tau\sim \nu$.
 The existence of such stopping times is guaranteed by Skorohod-type embeddings of probability measures in {\em ``subharmonic order,"} into Brownian motion. This problem is equivalent to an optimal mass transport problem with certain constraints, namely {\em the optimal subharmonic martingale transport}. Under the assumption of radial symmetry  on $\mu$ and $\nu$, we show that the optimal stopping time is a hitting time of a suitable barrier, hence is non-randomized and is unique.
\end{abstract}

\begin{keyword}[class=MSC]
\kwd[Primary ]{49-XX}
\kwd{60-XX}
\kwd[; secondary ]{52-XX}
\end{keyword}

\begin{keyword}
\kwd{Optimal Transport, Skorokhod Embedding, Monotonicity, Radial Symmetry.}
\end{keyword}
\tableofcontents
\end{frontmatter}





\section{Introduction}\label{intro} 


Let  $\mu$ and $\nu$ be two probability measures on $\R^d$, $d \geq 2$ with finite first moment, and let $(B_t)_t$ denote a Brownian motion with initial law $\mu$.  
We consider the following --possibly empty-- set of stopping times, with respect to the Brownian filtration:
\begin{align*} 
{\cal T}(\mu, \nu) = \{\tau \,|\, \text{$\tau$ is a stopping time, } B_0 \sim \mu, B_\tau \sim \nu, \text{ and } \E[\tau] < \infty \}, 
\end{align*} 
where here and in the sequel, the notation $X\sim \lambda$ means that the law of the random variable $X$ is the probability measure $\lambda$.  

For a cost function $c: \R^d \times \R^d \to \R$, we shall consider  the following  optimization problem
\begin{align}\label{opt}
 \text{ Maximize / Minimize }\ \E\, [c(B_0, B_\tau)] \quad \text{over} \quad \tau \in {\cal T}(\mu, \nu),
\end{align}  
provided of course ${\cal T}(\mu, \nu)$ is non-empty. 

Recall that a {\em stopping time} on a filtered probability space $(\Omega, \mathcal F, (\mathcal F_t)_t,\mathbb P)$ is a random variable $\tau :\Omega \to [0, +\infty]$ such that $\{\tau \leq t\}\in \mathcal F_t$ for every $t\geq 0$. A {\em randomized stopping time} is a probability measure $\tau$ on $\Omega \times [0, +\infty]$ such that for each $u \in \R_+$, 
the random time 
$\rho_u (\omega) :=\inf\{ t\ge 0: \tau_{\omega} ([0,t]) \ge u\}
$
 is a stopping time, where here $(\tau_{\omega)_w}$ is a disintegration of $\tau$ along the path ${\omega}$ according to $\P$, that is  $\tau(d\omega,  dt) = \tau_{\omega} (dt) P(d\omega)$. In the sequel, ``stopping time" will mean a randomized stopping time unless stated otherwise. We note that a stopping time is {\em non-randomized} if the disintegration $\tau_{\omega}$ is the Dirac measure on $\R_+$ for $\P$ a.e. $\omega$. 

In this paper, we will focus on cost functions of the form
\begin{align}\label{shcost}
c(x,y) = |x-y|^\alpha,
\end{align} 
where $0 < \alpha \neq 2$, though most results could apply to more general 
cost functions of the form 
$ c(x,y) = f(|x-y|),$
where $f : \R_+ \to \R$ is a continuous function such that $f(0)=0$. 
The purpose of this paper is to identify and characterize such  optimal stopping times, that is those where (\ref{opt}) is attained. In particular, we will be investigating when optimal stopping times are `true" stopping times, as opposed to randomized, and are therefore unique.


But first, we recall that the Skorokhod Embedding Problem (SEP) --which essentially asks for which pairs $(\mu, \nu)$, the set ${\cal T}(\mu, \nu)$ is non-empty-- was introduced by Skorokhod \cite{Sk65} in the early 1960s. Since then, the problem and its variants was investigated by a large number of researchers, and has  led to several important results in  
probability theory and stochastic processes. We refer to {\OB} \cite{Ob04} for an excellent survey of the subject, which describes no less than 
 $21$ solutions to (SEP). 
 
 More recently,  Hobson \cite{Ho98} made the connection between SEP and the robust pricing and hedging of financial instruments. Also,  Hobson-Klimmek \cite{HoKl12} and Hobson-Neuberger \cite{HoNe11} connected it to finding robust price bounds for the forward starting straddle. See also the excellent survey  of Hobson  \cite{Ho11}. 
 
The interest therefore shifted to the problem of finding optimal solutions among those that solve (SEP). In other words, which among the stopping times in  ${\cal T}(\mu, \nu)$ maximize or minimize a given cost function. 
 Among the multitude of contributions to these questions, we point to the papers of Beiglb{\"o}ck-Cox-Huesmann \cite{bch}, Cox-\OB-Touzi \cite{COT}, Dolinsky-Soner \cite{ds1, ds2}, Guo-Tan-Touzi \cite{GTT}, K{\"a}llblad-Tan-Touzi \cite{KTT}, to name just a few.
We do single out, however, the recent work of Beiglb{\"o}ck-Cox-Huesmann \cite{bch}, which used the analogy betweem the optimal SEP and the theory of optimal mass transport, 
to identify and prove the so-called {\em Monotonicity Principle} (MP), which will be one of the main tools used in this paper.

We note that most of the articles and surveys mentioned above deal only with the case when the marginals are measures on the real line. 
The case where the underlying spaces are higher dimensional is more delicate. This was illustrated in our paper  \cite{GKL2}, where we deal with general martingale mass transport. This paper deals with the somewhat related {\it optimal stopping problems} in higher dimensions.
What makes the study of SEP more involved in higher dimension is the fact that as soon as $d\geq 2$, there are many natural notions of convexity that are compatible with 
 higher dimensional Brownian motion, such as {\it subharmonicity and plurisubharmonicity}, notions that are more general and sometimes more subtle than convexity.

The optimization problem (\ref{opt}) obviously makes no sense unless ${\cal T}(\mu, \nu)$ is nonempty. Since $\{f(B_t), t\ge 0\}$ is a submartingale
for any subharmonic function $f$ on $\R^d$, we have that $\E f(B_0) \le \E f(B_\tau)$, and therefore 
\begin{equation}\label{order.0}
\int f \,d\mu \le \int f \,d\nu \hbox{ for every subharmonic function $f$},
\end{equation}
 which clearly makes it a necessary condition for ${\cal T}(\mu, \nu)$ to be nonempty. 
That this condition is also sufficient, has been the subject of many studies starting with the original work of Skorokhod in the one-dimensional case. We also refer to the work of Monroe \cite{Mo72}, Rost \cite{Ro71}, Chacon-Walsh \cite{ChWa76}, Falkner \cite{Fa80} and many others for refinements and different proofs. 

In Section 2, we shall give yet another proof of this fact, and show that every one-step {\it subharmonic martingale} $(X, Y)$, i.e.,  any couple of random variables verifying 
\begin{equation}
f(X) \le \E [ f(Y)|X] \quad \hbox{for every $f$ subharmonic,}
\end{equation}
can be realized by stopping Brownian motion, that is, there exists a --possibly randomized-- stopping time $\tau$ such that 
\begin{equation}
\hbox{$(B_0, B_{\tau})$ and $(X, Y)$  have the same joint distribution.}
\end{equation}
We shall use this fact in Section 3, to prove the following duality result:

Denote by ${\cal SH}(O)$ the cone of subharmonic functions on the open set $O$, and  consider the following cone of functions on $O\times O$.
\begin{align*}
 \mathcal{P}(O\times O) = \{ p \in C(\overline{O\times O})\ |\,  &  p(x,x) = 0 
 \text{ and}\, p(x, \cdot)\in {\cal SH}(O)  
  \text{ for all $x \in O$
  }\}.
\end{align*}
\begin{theorem}
\label{thm: no gap SEP}
Assume that $\mu, \nu$ are two compactly supported probability measures satisfying (\ref{order.0}),  such that $\supp \mu \cup \supp \nu$ is contained in an open set $O$.
If $c$ is a continuous cost on $O\times O$, then the following duality holds:
\begin{align*}
 \inf\left\{ \E\, [c(B_0, B_\tau)]; \, \tau \in {\cal T}(\mu, \nu)\right\}=\sup\left\{\int_O \beta d\nu - \int_O \alpha d\mu; \, (\beta, \alpha) \in {\cal K}_c(O)   \right\},
\end{align*}  
where
\begin{align*}
{\cal K}_c(O):= &\left\{\hbox{$\alpha, \beta: O \to \R$ locally Lipschitz, such that there is  $p \in \mathcal{P}(O\times O)$}\right. \\
  & \qquad \qquad \qquad \hbox{with } \ \beta (y) - \alpha(x) +  p(x, y)  \le c( x, y)  \ \forall x, y \in O \Big\}. 
\end{align*}
\end{theorem}
In order to prove this duality, we use the fact that (\ref{opt}) is actually equivalent to the {\em optimal subharmonic martingale problem}. Indeed, recall first that 
{\em the optimal martingale problem} consists in the following:
\begin{align}\label{eq:MT problem}
\mbox{Minimize}\quad \text{Cost}[\pi] = \iint_{O\times O} c(x,y)d\pi(x,y)\quad\mbox{over}\quad \pi\in {\rm MT}(\mu,\nu)
\end{align}
where MT$(\mu,\nu)$  is the set of {\it martingale transport plans,} that is, the set of probabilities $\pi$ on $O\times O$ such that 
\begin{itemize}
\item $\pi$ has marginals $\mu$ and $\nu$.
 \item For each  $\pi\in$ MT$(\mu,\nu)$, its disintegration $(\pi_x)_{x\in \R^d}$  w.r.t. $\mu$ is such that the barycenter of $\pi_x$ is at $x$. In other words, for any convex function $f$ on $O\subset \R^d$,
\begin{align}\label{eq:Jensen}
g(x) \le \int_O g(y)\,d\pi_x (y). 
\end{align} 
\end{itemize}
Problem \eqref{eq:MT problem} has been extensively studied, especially in one dimension by Beiglb{\"o}ck-Juillet \cite{BeJu16}, Beiglb{\"o}ck-Nutz-Touzi \cite{BeNuTo17}  and more recently in higher dimension by Ghoussoub-Kim-Lim  \cite{GKL2}.  Also, for recent developments on higher-dimensional decomposition of martingale transports, see De March-Touzi \cite{DeTo17} and \OB-Siorpaes \cite{ObSi17}.

In our situation, we need to consider the class of {\em subharmonic martingale transport plans}, that is the class 
 ${\rm SMT_O}(\mu,\nu)$  of those $\pi \in {\rm MT}(\mu, \nu)$ such that the inequality \eqref{eq:Jensen} also holds for every subharmonic function $f$ on $O$. This leads us to the {\em subharmonic  martingale optimal transport (SMOT)} problem: 
\begin{align}\label{eq:SMT_O problem}\mbox{Minimize}\quad \text{Cost}[\pi] = \iint_{O\times O}c(x,y)d\pi(x,y)\quad\mbox{over}\, \pi\in {\rm SMT_O}(\mu,\nu).
\end{align}
Since every convex function is subharmonic, ${\rm SMT_O}(\mu,\nu) \subset {\rm MT}(\mu,\nu)$, and for $d=1$ these two sets are the same. However, for $d\ge 2$, the inclusion is strict and  problems \eqref{eq:MT problem} and \eqref{eq:SMT_O problem} are different. 
Indeed, 
if $\mu$ is the uniform measure on the unit sphere in $\R^d$, and  $\nu$ has half of its 
mass at the origin and the other half uniformly distributed on the sphere of radius 2, then clearly, MT$(\mu,\nu) \neq \emptyset$.
On the other hand, if we consider the subharmonic function $g(z) = \log |z|\,$, then
\begin{align*}
\int g(z) \,d\mu \,(z) > - \infty = \int g(z)\, d\nu \,(z),
\end{align*}
which means that $SMT_O(\mu,\nu) = \emptyset$.

However, just like the convex case in one-dimension, we shall be able to show that 
\begin{align}
 \inf\left\{ \E\, [c(B_0, B_\tau)]; \, \tau \in {\cal T}(\mu, \nu)\right\}=\inf\left\{\iint_{O\times O}c(x,y)d\pi;\, \pi\in {\rm SMT_O}(\mu,\nu)\right\}.
 \end{align}

We then address our main question, which is to show that the optimal stopping time in (\ref{opt}) is not randomized and that it is therefore unique. The general case is still elusive, though we shall show in a forthcoming paper \cite{GKL3} that this conclusion will hold for general marginals $\mu, \nu$ under suitable conditions. In this paper, we shall focus on the case when $\mu$ and $\nu$ are radially symmetric, which has its own interest. 

 \begin{theorem}\label{thm:main}
Suppose $\mu$, $\nu$ are radially symmetric and compactly supported probability measures on $\R^d$, $d\geq2$. Assume $\mu \wedge \nu =0$ and $\mu(\{0\}) =0$. Let $c(x,y) = |x-y|^\alpha$, where $0 < \alpha \neq 2$. Then, the solution for the optimization problem \eqref{opt} under the marginals $\mu, \nu$ is unique and is given by a non-randomized stopping time.
\end{theorem}

We note that for the minimization problem and when $0<\alpha\leq1$, the assumption that $\mu \wedge \nu=0$ will not be necessary.

\section{Spherical martingales and Skorohod embedding}

In this section, we shall prove the following result, which belongs to the family of results around Skorohod embeddings. If  $\mu, \nu$ are two probability measures supported in a domain $O$ of $\R^d$, we shall sometimes use the following notation reminescent of Choquet theory,
$$
\mu \prec_O \nu \quad\hbox{if \quad $\int_O f d\mu \le \int_O f d\nu$ for every $f\in {\cal SH}(O)$.}
$$

\begin{theorem}\label{prop1} 
Let $\mu, \nu$ be two compactly supported probability measures on $\R^d$, and let  $O$ be an open set containing $\supp \mu \cup \supp \nu$. Then, the following properties are equivalent:
\begin{enumerate}
\item $\mu \prec_O \nu$. 

\item There exists a subharmonic martingale plan $\pi \in SMT_O(\mu, \nu)$. 

\end{enumerate}
Moreover, for every subharmonic martingale plan $\pi \in SMT_O(\mu, \nu)$, there exists a stopping time $\tau$ and a $d$-dimensional Brownian motion $(B_t)_t$ such that ${\rm Law}(B_0) = \mu$, ${\rm Law}(B_\tau) = \nu$, and $\pi$ is the joint distribution of $(B_0, B_\tau)$.\\
 Furthermore, $\tau = \tau \wedge \kappa$, where $\kappa$ is the first exit time of $B_t$ from $O$.
In other words, ${\cal T}_O(\mu, \nu)$ is nonempty. 
\end{theorem}

\begin{remark} The same result holds if $\mu, \nu$ are not compactly supported but under the condition that they have finite second moments. In this case the condition that $\mu \prec_O \nu$ should be replaced by 
$$\int_{\R^d} f d\mu \le \int_{\R^d} f d\nu \text{ \,\,for every $f\in {\cal SH}(\R^d)$ with } f(x) \le C_f (1+|x|^2).$$
In this case,  ${\cal T}_O(\mu, \nu)$ is again nonempty, and  there exists a stopping time $\tau$ and a $d$-dimensional Brownian motion $(B_t)_t$ such that ${\rm Law}(B_0) = \mu$, ${\rm Law}(B_\tau) = \nu$, and $(B_{\tau \wedge t})_{t \ge 0}$ is an $L^2$-bounded martingale.
\end{remark}
To prove  that 1) implies 2), we shall use the following variant of a classical theorem of Strassen \cite{St65}.
Let ${\rm Lip}(O)$ be the Banach space of bounded and Lipschitz functions on $O$. 
The set ${\cal P}(O)$ consisting of all Borel probability
measures on $O$, will be identified with a closed bounded convex
subset of the dual space ${\rm Lip}(O)^*$.  
We shall prove the existence of  a measurable map  $T\colon \ O\to {\cal P}(O)$
 such that 
 \begin{enumerate}[i)]
 \item  $\langle \nu, f\rangle = \int \langle T(x), f\rangle
 d\mu(x)$ for all  functions $f \in {\rm Lip}(O)$. 
 \item  $\delta_x \prec_{_{{\rm SH}(O)}} \pi_x$ for $\mu$-almost all $x\in O$.
  \end{enumerate}
 For that, define for
 each $f\in {\rm Lip}(O)$ the function 
 $$\hat f (x) = \inf \{ \varphi (x); \, -
 \varphi \in {\cal SH}(O)\, \hbox{\rm and $\varphi \ge f$}\}.$$
  On the vector space $S$ of all simple Borel functions $\theta \colon \
 O\to {\rm Lip}(O)$, define the sublinear functional $p\colon S\to {\R}$
 by $p(\theta) = \int \theta (x)\hat (x) d\mu(x)$. Let $S_0$ be the subspace
 of $S$ generated by the functions of the form ${\cal X}_B\otimes f$ defined by
\begin{equation*}
 ({\cal X}_B\otimes f)(x) =\left\{ 
 \begin{array}{lll}
f &{\rm if}& x\in B\\
 0&{\rm if}&x\notin B
 \end{array}\right.
  \end{equation*}
  Define on $S_0$ the linear functional $\ell ({\pmb 1} \otimes f) = \langle \nu,f\rangle$. The order $\mu \prec_{O}\nu$ implies that $\ell \le p$ on  $S_0$. Let $\tilde \ell$ be any Hahn-Banach extension of $\ell$ to
 the whole space $S$.  One can then check that the vector measure
 $$m\colon \ \Sigma \to {\rm Lip}(O)^*$$ defined on the Borel $\sigma$-field
 $\sum$ of $O$ by $\langle m(A),h\rangle = \tilde\ell({\cal X}_A\otimes h)$
 for all $A\in \Sigma$ and $h\in C_b(O)$, has average range in ${\cal
 P}(O)\colon$ that is ${m(A)\over \mu(A)} \in {\cal P}(O)$ for all
 $A\in \sum$.  Since the latter has the Radon-Nikodym property \cite{Ed85,Ed86}, $m$ has a density
 $T\colon \ O\to {\cal P}(O)$.  It is easy to check that $T$ is valued
 in ${\cal P}(O)$ and that it satisfies i) and ii).\\
 Define now the probability $\pi$ on $O\times O$ by $\pi(D) =
 \int T(x)(D_x) d\mu (x)$ where $D_x = \{ y\in O; (x,y) \in D)\}$. It is easy to verify that $\pi$ is a subharmonic martingale $\pi \in {\rm SMT}_O(\mu, \nu)$.\\


\noindent For the second part of the theorem, we introduce the notion of {\em spherical martingales}.
\begin{definition}
Let $U$ be the uniform probability measure on the unit sphere in $\R^d$. Let $(X_i)^{\infty}_{i=1}$ be i.i.d random variables on some probability space $(\Omega, \P)$, whose distribution is $U$. We define {\em spherical martingales} as follows:
\begin{align}
&F_0 = x \in \R^d\\
&F_n (X_1,...,X_n) - F_{n-1} (X_1,...,X_{n-1}) = r_n (X_1,...,X_{n-1}) \cdot X_n.
\end{align}
\end{definition}
In other words, at each time $n$, a particle splits uniformly onto a surrounding sphere of radius $r_n$.  A simple but important observation is that, the  push-forward measure of $\delta_x$ by a spherical martingale $F_n$  has the same law as $B^x_{\tau}$, where the stopping time $\tau$ is defined as follows: \\
$\tau_1$ is the first time $B^x$ hits the sphere $S(x, r_1)$ centered at $x$ and with radius $r_1$. If $B_{\tau_1} (\omega) = x_1 \in S(x, r_1)$, then define $\tau_2$ to be the first hitting time  $B^{x_1}$ hits the sphere $S(x_1,r_2)$. One can then define inductively a sequence of stopping times $\tau_1 \leq \tau_2 \leq ...$ such that $F_n (\P) = {\rm Law}(B^x_{\tau_n})$.

The following lemma is an analogue of a result of Bu-Schachermayer \cite[Proposition 2.1]{BuSc92} in the case of spherical martingales.
\begin{lemma}\label{shenvelope}
Let $O$ be an open set in $\R^d$ and $f : O \rightarrow \R \cup \{-\infty\}$ be an upper semicontinuous function. Define $f_0 = f$ and for $n \geq 1$
\begin{align*}
f_n (x) = \text{inf } \bigg\{ \int f_{n-1} (x + r y) \,dU(y) \bigg\}
\end{align*}
where the infimum is taken over all $r \geq 0$ such that $\{ x + r \overline{B}\} \subset O$. Then $(f_n)^{\infty}_{n=0}$ decreases pointwise to the largest subharmonic function $\hat{f}$ on $O$ dominated by $f$.
\end{lemma}
\begin{proof} First note that the sequence $(f_n)^{\infty}_{n=0}$ is decreasing. To show the upper-semicontinuity of $\hat{f}$, we proceed inductively and assume $f_{n-1}$ is upper semicontinuous. If $(x_k)^{\infty}_{k=0}$ in $O$ is such that $\lim_{k \rightarrow \infty} x_k = x_0$, and $r \geq 0$ is such that  $\{ x_0 + r \overline{B}\} \subset O$,  where $\overline{B}$ is the closed unit ball, then there is $k_0$ such that $\{ x_k + r \overline{B}\} \subset O$ for $k \geq k_0$. The upper semicontinuous function $f_{n-1}$ is bounded above on the relatively compact set 
$\cup^{\infty}_{k=k_0} \{ x_k + r \overline{B}\}$ and, for every $z \in \overline{B}$,
$f_{n-1} (x_0 + r z) \geq \limsup_{k \rightarrow \infty} f_{n-1} (x_k + r z)$. Hence by Fatou's lemma, 
\begin{align*}
 \int f_{n-1} (x_0 + r y) \,dU(y) \geq \limsup_{k \rightarrow \infty}  \int f_{n-1} (x_k + r y) \,dU(y) 
 \geq \limsup_{k \rightarrow \infty} f_n (x_k).
 \end{align*}
Thus $f_n (x_0) \geq \displaystyle\limsup_{k \rightarrow \infty} f_n(x_k)$, hence showing that $f_n$ and consequently $\hat f$ is upper-semicontinuous.\\
Let now $g$ be a subharmonic function on $O$ with $g \leq f$. Again, inductively, assuming  that $g \leq f_{n-1}$, then for $\{ x_0 + r \overline{B}\} \subset O$, 
\begin{align*}
\int f_{n-1} (x_0 + r y) \,dU(y) \geq  \int g (x_0 + r y) \,dU(y)\, \geq \,g(x_0),
\end{align*}
and so $f_n (x_0) \geq g (x_0)$. Hence $\hat{f} \geq g$. Finally, for $\{ x_0 + r \overline{B}\} \subset O$,  we get from monotone convergence
\begin{align*}
\hat{f} (x_0) = \lim_{n \rightarrow \infty} f_n (x_0) \leq \lim_{n \rightarrow \infty} \int f_{n-1} (x_0 + r y) \,dU(y) =  \int \hat{f} (x_0 + r y) \,dU(y).
\end{align*}
This shows that $\hat{f}$ is subharmonic, and the proof of the lemma is complete.
\end{proof}
We can now rewrite $f_n$ as follows:
\begin{align}\label{edgar}
f_n(x) = \inf \{\E [f(F_n)] : (F_i)^n_{i=0} \,\text{ is a spherical martingale in $O$ with } \, F_0 = x \}.
\end{align}

\begin{lemma}\label{denseness} Let $\mu, \nu$ be probabilities as in Theorem \ref{prop1} such that $\mu \prec_{O}\nu$, and let 
$\Phi$ be the following subset of $Lip^*(O)$, which is the space of all finite measures on $O$ with finite first moments,
\begin{align*}
\Phi = \{ F_n(\P) : (F_i)^n_{i=0} \,\hbox{is a spherical martingale valued in $O$ with $F_0\sim \mu$\}.}
\end{align*}
Set $\tilde \nu:= (1+|x|)\nu$ and 
$$
\tilde \Phi = (1+|x|)\Phi:= \{\tilde \sigma \,|\, \tilde \sigma = (1+|x|)\sigma \hbox{ for some  $\sigma \in \Phi$\}.}
$$
Then $\tilde \nu$ is in the weak$^*$-closure of $\tilde \Phi$ in $Lip^*(O)$.
\end{lemma}
\begin{proof}
Observe first that $\Phi$ is convex. Indeed, if $(F_i')^n_{i=0}$ and $(F_i'')^m_{i=0}$ are two spherical martingales, we may assume $n=m$, then define a spherical martingale $(F_i)^{n+1}_{i=0}$ by letting $F_0 = F_1 \sim \mu$ and for $1 \leq i \leq n$,
\begin{align*}
F_{i+1} (X_1,X_2,...,X_{i+1}) = 
\begin{cases}
F_i' (X_2,...,X_{i+1}) \text{\, if \,$X_1$ is in the upper hemisphere,}\\
F_i'' (X_2,...,X_{i+1}) \text{\, if \,$X_1$ is in the lower hemisphere.}
\end{cases}
\end{align*}
Clearly $F_{n+1} (\P) = \{F_{n}' (\P) + F_{n}'' (\P)\} / 2$ and hence $\Phi$ is convex, and therefore $\tilde \Phi = (1+|x|)\Phi$ is convex in $Lip^*(O)$.

If now the statement of the lemma were false, then by the Hahn-Banach theorem we can find a Lipschitz function $f$ on $O$ and real numbers $a < b$ such that
\begin{align*}
\int g \,d\nu \leq a, \, \text{while}\,  \int  g \circ F_n \,d\P  \geq b \,\text{ for every $F_n(\P) \in \Phi$},
\end{align*}
where $g(x) = (1+|x|)f(x)$. But since every element in $\Phi$ has initial distribution $\mu$, then by Lemma \ref{shenvelope} and \eqref{edgar}, we have $\int \hat{g} \,d\mu \geq b$ and therefore $a \geq \int g \,d\nu \geq \int \hat{g} \,d\nu \geq \int \hat{g} \,d\mu \geq b$, which is a contradiction.
\end{proof}
\begin{proof}[\bf Proof of Theorem~\ref{prop1}]
By Theorem \ref{denseness}, we have a sequence $\{\nu_n\} \subset \Phi$ such that $\int |x|^2 d\nu_n(x) \rightarrow \int |x|^2 d\nu(x)$. We know that $\nu_n = {\rm Law}(B_{\tau_n})$ for a sequence of stopping times ${\tau_n}$ and a Brownian motion $B$ with initial law $\mu$.  Hence in particular $\E \tau_n = \E |B_{\tau_n}|^2 =  \int |x|^2 d\nu_n(x) \leq V$ for some constant $V$ and for all $n$. The sequence $({\tau_n})$ is then tight, and it is standard that it has a convergent subsequence to a --possibly randomized-- stopping time $\tau$ in such a way that $\E f(B_{\tau_k}) \rightarrow \E f(B_\tau)$ for every $f$ continuous and bounded function on $\R^d$. (see for example \cite{BaCh74} or  \cite{bch}). In other words, ${\rm Law}(B_{\tau_k}) \rightarrow {\rm Law}(B_\tau)$, and therefore, $B_\tau \sim \nu$.
Note that $\tau_k = \tau_k \wedge \kappa$ by the definition of $\Phi$, and therefore $\tau$ also satisfies $\tau= \tau \wedge \kappa$. 
 Notice also that $\E \tau = \E |B_{\tau}|^2 \leq V$ as well, hence the martingale $(B_{\tau \wedge t})_{t \ge 0}$ is bounded in $L^2$. 

 Let now $\pi$ be a Subharmonic martingale plan and let $(\pi_x)_x$ be its disintegration w.r.t. $\mu$. Let $(B_t)_t$ be a Brownian motion with $B_0 \sim \mu$. Since $\delta_x\prec_O \pi_x$ for $\mu$-almost all $x$, and noting that $B_0$ and $B_t - B_0$ are independent, one can apply the above to select measurably a stopping time $\tau$ such that given $B_0$, we have ${\rm Law}(B_\tau \in \cdot \,|\, B_0) = \pi_{B_0} (\cdot)$. It is then clear that ${\rm Law} (B_0, B_\tau) = \pi$. 
\end{proof}
The following is now immediate.

\begin{corollary} \label{same} Assume that $\mu, \nu$ are compactly supported, and that $O$ is an open set containing $\supp \mu \cup \supp \nu$. If $\mu \prec_{O}\nu$, then,
\begin{align*}
 \inf\left\{ \E\, [c(B_0, B_\tau)]; \, \tau \in {\cal T}(\mu, \nu)\right\}=\inf\left\{\iint_{O\times O}c(x,y)d\pi;\, \pi\in {\rm SMT}(\mu,\nu)\right\}.
 \end{align*}
\end{corollary}




\section{Weak duality for the optimal Skorohod embedding problem}

From now on, we will assume that the prescribed marginals $\mu$ and $\nu$ are supported in a bounded open ball $O \subset \R^d$. Therefore every measure appearing in the sequel is also supported in $O$, and for every stopping time $\tau$,
\begin{align*}
\text{$\tau = \tau \wedge \kappa$, where $\kappa$ is the first exit time for Brownian motion from $\partial O$. }
\end{align*}
In other words, no one can escape the ``universe" $O$. This assumption is made due to the fact that we deal with the cost \eqref{shcost} for every $\alpha > 0$, and it should be sufficient to assume appropriate moment conditions on the marginals $\mu, \nu$ according to $\alpha$ in \eqref{shcost}. For example, if $\alpha \le 2$ then $\mu, \nu$ can only  be  assumed to have finite second moments and the same arguments will work without significant modification.

In order to prove Theorem \ref{thm: no gap SEP}, we first prove the weak duality for the subharmonic martingale optimal transport problem. For that, we start by characterizing such martingales in terms of their ``orthogonality" vis-a-vis the cone ${\cal P}(O\times O)$.
\begin{lemma}\label{lem: char P S}
 Suppose $\pi \in {\rm Prob}(O \times O)$ is such that its first marginal $\pi^1$ is absolutely continuous with respect to Lebesgue measure.  Then, $\pi$ is a subharmonic martingale transport plan on $O$ if and only if
 \begin{equation}\label{pos}
 \int_U p(x, y)  d\pi (x, y) \ge 0  \quad \forall p \in \mathcal{P}(O\times O) \text{ and for a.e. $x \in O$}.
 \end{equation}
\end{lemma}
\begin{proof} Note first that if $\pi \in {\rm SMT}(O)$, then for any $p \in \mathcal{P}(O\times O)$ we have, thanks to the subharmonicity of  of $y \mapsto  p(x, y)$, that 
\begin{align*}
 \int p(x, y) d\pi (x, y) = \int p (x, y) d\pi_x (y) d\pi^1(x)  \ge  \int  p(x, x) d\pi^1 (x) =0.
 \end{align*}
 For the reverse, 
 we assume that (\ref{pos}) holds, and consider functions of the form 
\begin{align*}
 p(x, y) = \psi^z_\epsilon (x) (\phi (y) - \phi(x)), 
\end{align*}
   where $0\le \psi^z \in C^1 (O)$, $\psi^z_\epsilon \to \delta_z $ in $L^1$ as $\epsilon \to 0$, and $\phi$ is an arbitrary subharmonic function on $O$.  
Clearly $p(x, y) \in \mathcal{P}(O\times O)$, 
  therefore, 
\begin{align*}
 0 \le \int p(x, y) d\pi (x, y) = \int \psi^z_\epsilon (x) \left(\int \phi (y) d\pi_x (y)  - \phi (x) \right) d\pi^1(x). 
\end{align*}
Take $\epsilon \to 0$, and see that for each Lebesgue density point $z$ of both the measurable function 
$x \mapsto \int \phi (y) d\pi_x (y)  - \phi (x)$ and the measure $\pi^1$, 
\begin{align*}
 0 \le \int \phi (y) d\pi_z (y)  - \phi (z).
\end{align*}
Since $\pi^1$ is absolutely continuous, this shows that $\pi \in {\rm SMT}(O)$ as desired, completing the proof. 
\end{proof}
\noindent With this lemma at hand, we can prove the following duality by using a standard argument. 
\begin{proposition}
\label{thm: no gap HMT}
Assume that $\mu, \nu$ are compactly supported in a bounded open set $O$ such that $\mu$ is absolutely continuous with respect to Lebesgue measure and $\mu\prec_O\nu$.
If $c$ is a continuous cost on $O\times O$, then the following duality holds:
\begin{align*}
\inf \left\{ \int_{O\times O} c(x, y) d\pi \ | \ \pi \in {\rm SMT}(\mu, \nu) \right\} 
=\sup\left\{\int_O \beta d\nu - \int_O \alpha d\mu; \, (\beta, \alpha) \in {\cal K}_c(O)   \right\},
\end{align*}  
where
\begin{align*}
{\cal K}_c(O):= &\left\{\hbox{$\alpha, \beta: O \to \R$ locally Lipschitz, such that there is  $p \in \mathcal{P}(O\times O)$}\right. \\
  & \qquad \qquad \qquad \hbox{with } \ \beta (y) - \alpha(x) +  p(x, y)  \le c( x, y)  \ \forall x, y \in O \Big\}. 
\end{align*}
 \end{proposition}
\begin{proof} Clearly, the right-hand side is smaller than the left-hand side since for every $\alpha, \beta$ in ${\cal K}_c(O)$ and their corresponding $p$ in $\mathcal{P}(O\times O)$, we have $\int p(x, y) d\pi (x, y) \geq 0$, and $\pi$ has marginals $\mu$ and $\nu$.

For the reverse inequality, we first note that Kantorovich duality for the standard optimal transport problem with cost $c(x, y) - p(x, y)$, yields that 
\begin{align*}
  \sup\left\{\int_O \beta d\nu - \int_O \alpha d\mu; \, (\beta, \alpha) \in {\cal K}_c(O)   \right\}
  = \sup_{p \in  \mathcal{P} }  \inf_{\pi \in \Gamma(\mu, \nu) } \int (c(x, y) - p(x, y) )d\pi. 
  \end{align*}
Now, by Von-Neuman min-max theorem, we can interchange the order of inf and sup 
(note that $X= \Gamma(\mu, \nu)$ is compact convex in ${\rm Prob}(\R^d \times \R^d)$ and $Y=\mathcal{P}_{\mathcal{SH}}$ is convex in $L^1 (\overline{\Omega\times \Omega)}$) to get that
  \begin{align}\label{eq: inf sup 2}
\sup\left\{\int_O \beta d\nu - \int_O \alpha d\mu; \, (\beta, \alpha) \in {\cal K}_c(O)   \right\}= \inf_{\pi \in \Gamma(\mu, \nu)} \sup_{p \in \mathcal{P}} \int (c(x, y) - p(x, y)) d\pi.  
 \end{align}
Now, notice that the supremum over $p$ in the latter can be finite only when $\int p(x, y) d\pi (x, y) = 0$, since otherwise we can replace $p$ with $\lambda p$ for some $\lambda \ne 0$, making the value of the integral $\int (c(x, y) - \lambda p(x, y) ) d\pi (x, y)$ as large as possible. Therefore,  from Lemma~\ref{lem: char P S}, the infimum in \eqref{eq: inf sup 2} can be restricted  to $\pi \in {\rm SMT}(\mu, \nu)$, that is, 
\begin{align*}
\sup\left\{\int_O \beta d\nu - \int_O \alpha d\mu; \, (\beta, \alpha) \in {\cal K}_c(O)   \right\}  = \inf_{\pi \in {\rm SMT}(\mu, \nu) } \sup_{p \in \mathcal{P}} \int (c (x, y) - p (x, y) )d\pi.
\end{align*}
But the last expression is greater than or equal than 
$$\inf \left\{ \int_{O\times O} c(x, y) d\pi \ | \ \pi \in {\rm SMT}(\mu, \nu) \right\},$$
 since $0 \in \mathcal{P}(O\times O)$. 
 This completes the proof of the weak duality for submartingale transport plans. 
\end{proof}
The proof of Theorem \ref{thm: no gap SEP} follows from the above combined with Proposition \ref{same}.


\section{A monotonicity principle and its symmetric version} 

We first make more precise the filtered Brownian probability space on which we operate.  We consider 
 $C(\R_+) = \{ \omega : \R_+ \to \R^d \,\,|\,\, \omega(0)=0, \,\omega \text{ is continuous}\}$ to be the  path space starting at $0$. The probability space will be $\Omega := C(\R_+) \times \R^d$ equipped with the probability measure $\mathbb P := \mathbb W \otimes \mu$, where $\mathbb W$ is the Wiener measure and $\mu$ is a given (initial) probability measure on $\R^d$. Stopping times and randomized stopping times will be with respect to this obviously filtered probability space.

Following 
Beiglb{\"o}ck-Cox-Huesmann \cite{bch}, we let each $(\omega, x) \in \Omega$ denotes a path $(\omega, x) (t) = \omega^x (t) := x + \omega(t)$ starting at $x\in \R^d$. We then let $S$ be the set of all {\em stopped paths}
\begin{align}\label{StoppedPaths}
S =\{(f^x, s) \,\,|\,\, f^x:[0,s] \to \R^d \mbox{ is continuous and $f^x(0)=x$}\}.
\end{align} 
Next, we introduce the \emph{conditional randomized stopping time given $(f^x,s) \in S$}, that is, the normalized stopping measure given that we followed the path $f^x$ up to time $s$. For $(f^x, s) \in S$ and $\omega \in C(\R_+)$, define the {\em concatenate path} $f^x\oplus\omega$ by $(f^x\oplus\omega)(t) = f^x(t)$ if $t\leq s$, and $(f^x\oplus\omega)(t) = f^x(s) + \omega (t-s)$ if $t>s$.
\begin{definition}[Conditional stopping time \cite{bch}]\label{def:CRST}
 The conditional randomized stopping time of $\xi$ given $(f^x,s)\in S$, denoted by $\xi^{(f^x, s)}$, is defined on $\omega \in C(\R_+)$ as follows:
\begin{align*}\label{CondStop} &\xi^{(f^x,s)}_\omega([0,t])= \frac{1}{1 - \xi_{f^x\oplus\omega}([0,s])}\left(\xi_{f^x\oplus\omega}([0,t+s])-\xi_{f^x\oplus\omega}([0,s])\right) \\ &\mbox{\,\,\,\quad\quad\quad\quad\quad \quad\quad if \quad$\xi_{f^x\oplus\omega}([0,s]) < 1$,}\\
&\xi^{(f^x,s)}_\omega(\{0\}) = 1\quad\mbox{if \quad $\xi_{f^x\oplus\omega}([0,s]) =1$.}
\end{align*}
\end{definition}
\noindent According to \cite{bch}, this is the normalized stopping measure of the ``bush'' which follows the ``stub'' $(f^x,s)$. Note that $\xi_{f^x\oplus\omega}([0,s])$ does not depend on the bush $\omega$.

We now give a heuristic description of a notion that we introduce below. Let $\tau$ be a stopping time, and suppose that Brownian motion starts at $x$, reaches $y$ at time $t$, then continuous to travel until the stopping time $\tau$,  i.e. $B^x_t(\omega) = y$ and $t < \tau(\omega^x)$. 
Now by the strong Markov property, Brownian motion will continue after it hits $y$ for the remaining time $\tau - t$ leading to a  (conditional) probability measure $\psi_y = {\rm Law}(B^y_{\tau - t})$ that is centered at $y$ .
Suppose now that Brownian motion starting at ${x'}$ is stopped at $\tau$, and let   $y'=B^{x'}_{\tau} (\omega')$. We shall then denote such a pair of transport along the stopping time $\tau$ by
$(x \to  \psi_y : x' \to y') \in \tau$.
We shall then consider the following cost for such a pair of transport
\begin{align}
C (x \to \psi_y : x' \to y') = \int c(x, z) \,d\psi_y (z) + c(x', y').
\end{align}
More formally, consider $S$ the set of all stopped paths, and let  $\Gamma \subset S$ be a concentration set of a given stopping time $\tau$, i.e. $\tau(\Gamma) =1$. We shall write
\begin{align}
(x \to \psi_y : x' \to y') \in (\tau, \Gamma), 
\end{align}
if there exist $(f^{x}, t)$, $(g^{x'}, t') \in \Gamma$  and  $s < t$, such that $y = f^{x}({s})$, $y' = g^{x'}({t'})$, and $\psi_y =  {\rm Law}(B^y({\tau^{(f^{x},s)}}))$.

In the sequel, we shall denote by ${\rm SH}(x)$ any probability measure $\psi$ whose barycenter $x$ satisfies $\delta_x \prec_O \psi$. As seen above, these are the probabilities that can be obtained by stopped Brownian motions starting at $x$. The following proposition will be crucial for our main result. It was proved  by Beiglb{\"o}ck-Cox-Huesmann \cite{bch} for even more general path-dependent costs. 

\begin{theorem}[Monotonicity principle \cite{bch}]\label{monoprin} Suppose $c$ is a cost function and that $\tau$ is an optimal stopping time for the minimization problem \eqref{opt}. Then, there exists a  Borel set $\Gamma\subseteq S$ such that $\tau(\Gamma)=1$, and $\Gamma$ is {\em $c$-monotone} in the following sense: If $\psi_y \in {\rm SH}(y)$ and $(x \to \psi_y : x' \to y) \in (\tau,\Gamma)$, then
\begin{align}\label{sh5}
C(x \to \psi_y : x' \to y) \leq C(x \to y : x' \to \psi_{y}).
\end{align}
\end{theorem}
Here is a first simple application of the monotonicity principle.
\begin{corollary} \label{stay} Let $\tau$ be a stopping time in ${\cal T}_O(\mu, \nu)$ that minimizes problem \eqref{opt} and assume $c(x,y) = |x-y|^\alpha$ with $0<\alpha\leq1$. Then, we must have $\tau=0$ 
on the common mass $\mu \wedge \nu$. 
 Therefore, the minimization problem \eqref{opt} can be restricted to the disjoint marginals $\bar \mu := \mu - \mu \wedge \nu$ and $\bar \nu := \nu - \mu \wedge \nu$.
\end{corollary}
\begin{proof}
Indeed, if not, then heuristically, there is a particle at $x$ in $\mu \wedge \nu$ which diffuses to some $\psi_x \in {\rm SH}(x) \setminus \{\delta_x\}$, and so another particle at $y$, $y \neq x$, has to flow into $x$ and make up for the loss. Mathematically, for any $\Gamma \subset S$ with $\tau(\Gamma)=1$, there should exist $x, y \in \R^d$, $x \neq y$ and $ \psi_x \in {\rm SH}(x) \setminus \{\delta_x\}$, such that
\begin{align*}
(x \to \psi_x : y \to {x}) \in (\tau, \Gamma).
\end{align*}
Now for $0<\alpha\leq1$ and $z \in \R^d$, $|x-z|^{\alpha} + |y-x|^{\alpha} \geq |y-z|^{\alpha}$, so that
\begin{align}
\int |x-z|^{\alpha} d\psi_x (z) + |y-x|^{\alpha} \geq \int |y-z|^{\alpha} d\psi_x (z).
\end{align}
In other words,
\begin{align}
C(x \to \psi_x : y \to {x}) \geq C(x \to x : y \to \psi_{x}).
\end{align}
But the inequality is in fact strict since $\psi_x \neq \delta_{x}$, hence by theorem \ref{monoprin}, $\tau$ cannot be a minimizer.
\end{proof}

\begin{remark} Note that the above corollary implies that whenever we are studying the minimization problem with the cost $c(x,y) = |x-y|^\alpha$ with $0<\alpha\leq1$, we can assume that $\mu \wedge \nu=0$ without loss of generality.
\end{remark}
We now give a variant of Theorem \ref{monoprin}, which exploits the radial symmetry of the marginals $\mu$ and $ \nu$. For this, we will introduce several notions related to radial symmetry. \\
First, recall that a stopping time $\xi$ can be considered as a probability measure on  the set $S$ of stopped paths. We interpret $\xi$ as a transport plan in the following way: for $(f^x,s) \in S$, $\xi$ transports the infinitesimal mass $d\xi \big{(}(f^x,s)\big{)}$ from $x \in \R^d$ to $f^x(s) \in \R^d$  along the path $(f^x, s)$. Now define a map $T : S \to \R^d \times \R^d$ by 
$$T((f^x,s)) = (x, f^x(s))$$
and let $T\xi$ be the push-forward of the measure $\xi$ by the map $T$, thus $T\xi$ is a probability measure on $\R^d \times \R^d$. Now we give the definition of $R$-equivalence.

\begin{definition} Let $\lambda(x)=|x|$ be the modulus map. 
\begin{enumerate}
\item Two probability measures $\phi$ and $\psi$ on $\R^d$ are said to be $R$-equivalent if their push-forward measures by $\lambda$ coincide, i.e. $\lambda_\# \phi = \lambda_\# \psi$.  We then write $\phi \cong_R \psi$.
\item Two stopping times $\xi$ and $\zeta$ are said to be $R$-equivalent if the first marginals and second marginals of $T\xi$ and $T\zeta$ are $R$-equivalent, respectively. We then write $\xi \cong_R \zeta$.
\end{enumerate}
\end{definition}
\noindent The following symmetrization was introduced in \cite{Lim} for the {\em Martingale Optimal Transport Problem}. It will also be useful in this paper. \\
Let $\mathfrak{M}$ be the group of all $d \times d$ real orthogonal matrices, and let $\mathcal{H}$ be the Haar measure on $\mathfrak{M}$. For a given $M \in \mathfrak{M}$ and a stopping time $\xi$, we define $M\xi$ as follows: for each $A \subset S$, set
$$(M\xi)(A) = \xi(M(A)).$$
Clearly, $M\xi$ is also a stopping time. Now we introduce the symmetrization operator which acts on both the probability measures on $\R^d$ and on the stopping times.
\begin{definition}\label{symmetrization} The symmetrization operator $\Theta$ acts on  the set of probability measures on $\R^d$, and on the set of stopping times as follows:
  \begin{enumerate}
\item 
For each probability measure $\mu$ on $\R^d$ and $B \subset \R^d$,
$$(\Theta\mu)(B) =\int_{M \in \mathfrak{M}}  (M\mu)(B) \,d\mathcal{H}(M).$$
\item 
For each stopping time $\xi$ and $A \subset S$,
$$(\Theta\xi)(A) =\int_{M \in \mathfrak{M}}  (M\xi)(A) \,d\mathcal{H}(M).$$
\end{enumerate}
\end{definition}
Observe that $\Theta\mu$ is the unique radially symmetric probability measure which is $R$-equivalent to $\mu$. Moreover, for any stopping time $\xi$, notice that
\begin{align}
\text{if $T\xi$ has marginals $\mu$ and $ \nu$, then $T(\Theta\xi)$ has marginals $\Theta\mu$ and $ \Theta\nu$.}
\end{align}
This leads to the following important observation: Assume $c(x,y)$ is a rotation invariant cost function, i.e., $c(Mx, My) = c(x,y)$ for any $M \in \mathfrak{M}$,  and define the cost of a stopping time $\xi$ to be:
$$C(\xi) = \int_{\R^d \times \R^d} c(x,y)\,T\xi(dx,dy).$$
If $\xi$ solves the minimization problem \eqref{opt} where $T\xi$ has radially symmetric marginals $\mu$ and $\nu$, then for any stopping time $\zeta$, we must have
\begin{align}\label{compare}
C(\zeta) \ge C(\xi) \quad \text{whenever} \quad \zeta \cong_R \xi.
\end{align}
Indeed, if $C(\zeta) < C(\xi)$, then $\Theta\zeta$ will solve the minimization problem \eqref{opt} with the same marginals $\mu, \nu$ and less cost, a contradiction. Of course, if $\xi$ solves the maximization problem then the opposite inequality in \eqref{compare} must hold.\\

We are now ready to introduce the radial monotonicity principle. 
\begin{proposition}[Radial monotonicity principle]\label{rmonoprin} Suppose that $c$ is rotation invariant cost function and that $\tau$ is an optimal stopping time for the corresponding minimization problem \eqref{opt}, where the marginals $\mu, \nu$ are radially symmetric. Then, there exists a  Borel set $\Gamma\subseteq S$ such that $\tau(\Gamma)=1$, and $\Gamma$ is radially $c$-monotone in the following sense: If $\psi_y \in {\rm SH}(y)$, $(x \to \psi_y : x' \to {y'}) \in (\tau, \Gamma)$ and $|y| = |y'|$, then
\begin{align}\label{sh6}
C(x \to \psi_y : x' \to {y'}) \leq C(x \to y : x' \to \phi_{y'}) 
\end{align}
for any $\phi_{y'} \in {\rm SH}(y')$ that is R-equivalent to $\psi_y$.

\end{proposition}
\begin{proof} It follows directly from the general monotonicity principle once applied to the case where the marginals are radially symmetric.  Indeed, if the optimal stopping time $\tau$ allows a particle starting at $x$ to diffuse  when it reaches $y$ so that  it becomes the probability measure $\psi_y$, but takes another particle at ${x'}$  to stop at $y'$, and if we have the opposite inequality 
\begin{align}\label{sh2}
C(x \to \psi_y : x' \to {y'}) > C(x \to y : x' \to \phi_{y'}) 
\end{align}
for some $\phi_{y'} \in {\rm SH}(y')$, then we can ``modify" the stopping time $\tau$ to $\tau'$ in such a way that  the particle at $x$ now stops at $y$ by $\tau'$, but instead the particle at $x'$ starts diffusing at $y'$ until it becomes  $\phi_{y'}$. Then \eqref{sh2} means that the cost for $\tau'$ is smaller than that of $\tau$, but note that  the modified stopping time $\tau'$ may not satisfy the terminal marginal condition $\nu$. However, as $\psi_y \cong_R \phi_{y'}$ and $|y| = |y'|$, $\tau'$ is also $R$-equivalent to $\tau$, a contradiction by \eqref{compare}.  The radial monotonicity principle asserts the existence of a set of stopped paths $\Gamma$ which supports the optimal stopping time and resists any such modification.
\end{proof}

\section{A variational lemma}
Let $S_{y,r}$ be the uniform probability measure on the sphere of center $y$ and radius $r$, arguably the most simple element in ${\rm SH}(y)$. We choose $c(x,y) = |x-y|$ for simplicity and consider the following ``gain" function,
\begin{align*}
G(x) = G(x,y,r) := \int  |x-z| \,dS_{y,r}(z) -  |x-y|,
\end{align*}
and its gradient $\nabla_x G$. Note that $G$ is essentially the increase in cost when the Dirac mass at $y$ diffuses uniformly onto the sphere. It satisfied the following properties:
\begin{enumerate}
\item If $|x-y| < |x'-y'|$ and $r$ is fixed, then $G(x,y,r) > G(x',y',r)$.\\
In other words, for the same diffusion, the increase in cost is greater when the distance $|x-y|$ is small. Hence, for the minimization problem, it is better to stop particles near the source $x$, and let the particles that are far from $x$ to diffuse, as long as the given marginal condition is respected.
\item The vector $\nabla G(x)$ points toward the direction $y-x$, thus the directional derivative $\nabla_u G(x)$ is
\begin{align*}
\nabla_u G(x) < 0 \quad \text{if} \quad \langle u, y-x \rangle < 0.
\end{align*}
Hence the gain function decreases when $x$ moves away from the ``center of diffusion" $y$.  By combining this with the radial monotonicity principle, one can get crucial information about the optimal stopping time. 
\end{enumerate}
From now on,  $c(x,y) = |x-y|^{\alpha}$, $0 < \alpha \neq 2$. We define the {\em gain function}
\begin{align*}
G(x, \psi_y) := \int  |x-z|^\alpha d\psi_y (z) -  |x-y|^\alpha \quad \text{for} \quad \psi_y \in {\rm SH}(y),
\end{align*}
and note that 
\begin{align*}
G(x, \psi_y) - G(x', \phi_{y'}) = C(x \to \psi_y : x' \to {y'}) - C(x \to y : x' \to \phi_{y'}).
\end{align*}
The following variational lemma is crucial for our analysis.
\begin{lemma}\label{key}
Let $x, y$ be nonzero vectors in $\R^d$ and let $r = |x|$. Let $u$ be a unit tangent vector to the sphere $S_r$ at $x$, such that $\langle u, y \rangle <0$. Let $\psi_y \in {\rm SH}(y) \setminus \{ \delta_y\}$. Then, there exists a $\phi_y \in {\rm SH}(y)$ which is R-equivalent to $\psi_y$, such that
\begin{align*}
G(x, \phi_y) &= G(x, \psi_y) \\
\nabla_u G(x, \phi_y) &< 0 \quad \text{if} \quad \,0 < \alpha < 2,\\
\nabla_u G(x, \phi_y) &> 0 \quad \text{if} \quad\quad\quad \alpha > 2,
\end{align*}
where the directional derivative is applied to the $x$ variable.
\end{lemma}
\begin{proof}[\bf Proof of Lemma~\ref{key}]
Let $c(x,z) = |x-z|^{\alpha}$ and take its partial derivative $\frac{\partial}{\partial {x_d}}$ at $x=0$, to obtain
\begin{align}\label{sh9}
h(z) := \nabla_{e_d} c (x,z) \big|_{x=0} = - \alpha \, |z|^{\alpha - 2}\, z_d
\end{align}
Taking the Laplacian in $z$, we get
\begin{align}\label{laplacian}
\Delta \,h(z) = -\alpha (\alpha - 2) (\alpha + d - 2) \,|z|^{\alpha -4} \, z_d.
\end{align}
In other words,  the function $h$ is
\begin{enumerate}[i)]
\item strictly superharmonic in the lower half-space $\{z_d < 0\}$ if $0 < \alpha <2$
\item strictly subharmonic in the lower half-space $\{z_d < 0\}$ if  $\alpha >2.$
 \end{enumerate}
Let $0 < \alpha < 2$, and assume without loss of generality that $x=x_1\, e_1=(x_1,0,...,0)$ and $u = e_d$. Then $\langle u, y \rangle <0$, which means that $y$ is in the lower half-space. Let $H=\{ z \in \R^d : z_d = 0\}$ and choose a stopping time $\tau$ for the Brownian motion $B^y$ such that $ {\rm Law}(B^y_\tau) = \psi_y$, and let $\eta$ be the first time $B^y$  hits $H$. Let $\sigma_y =  {\rm Law}(B^y_{\tau \wedge \eta})$. Then $\sigma_y$ is supported in the lower half-space and it is nontrivial, hence by the strict superharmonicity \eqref{laplacian}, we have
\begin{align*}
\nabla_u G(x, \sigma_y) < 0.
\end{align*}

Now we modify $\tau$ to $\tau'$ in the following way; if $\tau \leq \eta$, then we let $\tau' = \tau$. But if $\tau > \eta$, in other words if a particle at $y$ has landed on $H$ but not completely stopped by $\tau$, then we symmetrize the remaining time of $\tau$ (i.e. the conditional stopping time) with respect to $H$ and get $\tau'$. More precisely, let $\tau_H$ be the reflection of the  conditional stopping time of $\tau$ with respect to  $H$; that is, if $\tau$ stops a path emanating from $H$, then $\tau_H$ stops the reflected path at the same time. Now define $\tau' := \frac{\tau+\tau_H}{2}$ to be a randomization; before re-starting Brownian motion on $H$, we flip a coin and choose either $\tau$ or $\tau_H$ for the conditional stopping time.

Now, define $\phi_y =  {\rm Law}(B^y_{\tau'})$ and observe that
\begin{enumerate}[i)]
\item $\phi_y \cong_R \psi_y$ by the symmetry with respect to $H$ in the definition of $\tau'$.
\item $\nabla_u G(x, \phi_y) = \nabla_u G(x, \sigma_y)$, since the function $z \mapsto \nabla_{e_d} c (x,z)$ is odd in $z_d$ (see \eqref{sh9}) hence the symmetrization in the definition of $\tau'$ does not change $\nabla_u G$.
\end{enumerate}
This proves the lemma.
\end{proof}
Here is a consequence of the variational lemma.

\begin{proposition}\label{consequence}
Let $x_0, x_1, y$ be nonzero vectors in $\R^d$ and let $r = |x_0| = |x_1|$. Assume $|x_0 - y| < |x_1 - y|$. Fix a measure $\phi_y \in {\rm SH}(y) \setminus \{\delta_y\}$ and define
\begin{align*}
 \underline{\Re}(x, \phi_y) &:= \bigg{\{}\psi_y \in {\rm SH}(y) : \psi_y \in \arg\displaystyle\min_{\sigma_y \cong_R \phi_y}  \int  |x-z|^\alpha d\sigma_y (z)\bigg{\}}, \\
 \overline{\Re}(x, \phi_y) &:= \bigg{\{}\psi_y \in {\rm SH}(y) : \psi_y \in \arg\displaystyle\max_{\sigma_y \cong_R \phi_y}  \int  |x-z|^\alpha d\sigma_y (z)\bigg{\}}, \\
\underline{G} (x) &:= G (x, \psi_y) =  \int  |x-z|^\alpha d\psi_y (z) - |x - y|^\alpha \quad \text{for any} \quad \psi_y \in \underline{\Re}(x, \phi_y),\\
\overline{G} (x) &:= G (x, \psi_y) =  \int  |x-z|^\alpha d\psi_y (z) - |x - y|^\alpha \quad \text{for any} \quad \psi_y \in \overline{\Re}(x, \phi_y).
\end{align*}
Then, 
\begin{align*}
&\underline{G} (x_0) > \underline{G} (x_1) \quad \text{and} \quad \overline{G} (x_0) > \overline{G} (x_1)   \quad \text{if} \quad 0 < \alpha <2,\\
&\underline{G} (x_0) < \underline{G} (x_1) \quad \text{and} \quad \overline{G} (x_0) < \overline{G} (x_1)   \quad \text{if} \quad\quad\,\,\,\, \, \alpha > 2.
\end{align*}
\end{proposition}
\begin{proof}
First, we claim that $\underline{G}(x)$ and $\overline{G}(x)$ are continuous.\\
Indeed, let $x_n \rightarrow x$ in $\R^d$, and define
\begin{align*}
\underline{C} (x) := \underline{G} (x) + |x-y|^\alpha =  \int  |x-z|^\alpha d\psi_y (z)\quad \text{for any}\quad \psi_y \in \underline{\Re}(x, \phi_y).
\end{align*}
Set $a_n = \underline{C} (x_n)$ and $a = \underline{C} (x)$. 
Choose any subsequence $\{a_k\}$ of $\{a_n\}$ and  a corresponding sequence of measures $\psi_k \in \underline{\Re}(x_k, \phi_y) $. By compactness, $\{\psi_k\}$ has a subsequence $\{\psi_l\}$ which converges to, say $\psi$. Note that $\psi \in {\rm SH}(y)$ and $\psi \cong_R \phi_y$ by weak convergence. Now, write
\begin{align*}
\int  |x_l-z|^\alpha d\psi_l (z) - \int  |x - z|^\alpha d\psi (z) 
&= \bigg[ \int  (|x_l-z|^\alpha - |x - z|^\alpha) \,d\psi_l (z) \bigg] \\
&\quad + \bigg[\int  |x - z|^\alpha d\psi_l (z) - \int  |x - z|^\alpha d\psi (z) \bigg].
\end{align*}
The first bracket goes to zero as $l \rightarrow \infty$ since $|x_l-z|^\alpha - |x - z|^\alpha \rightarrow 0$ uniformly on every compact set in $\R^d$, and the second bracket goes to zero since $\psi_l \rightarrow \psi$.\\
Now we claim that 
\begin{align}
\int  |x - z|^\alpha d\psi (z) = \underline{C} (x) = a, \quad \text{i.e.}\quad \psi \in \underline{\Re}(x, \phi_y).
\end{align}
If not, then there exists a $\rho \in \underline{\Re}(x, \phi_y)$ such that 
\begin{align*}
\int  |x - z|^\alpha d\psi (z) &> \int  |x - z|^\alpha d\rho (z), \text{ hence} \\ 
\int  |x_l - z|^\alpha d\psi_l (z) &> \int  |x_l - z|^\alpha d\rho (z) \text{ \,for all large $l$,}
\end{align*}
a contradiction since $\psi_l \in \underline{\Re}(x_l, \phi_y)$. Therefore, $a_n \rightarrow a$, since this holds for any subsequence. 

To complete the proof of the proposition, we let $0 < \alpha < 2$, and choose any differentiable curve $x(t): [0, 1] \rightarrow S_r$ with $x(0) = x_0, \,x(1) = x_1$ and such that $|\,x(t) - y\,|$ is strictly increasing. In other words, we choose $x(t)$ in such a way that $|x(t)|=r$ and $\langle \frac{d}{dt}x(t), y \rangle <0$ for all $t$. Now for a fixed $t \in [0,1]$, choose any $\psi_y \in \underline{\Re}(x(t), \phi_y) $. Then Lemma \ref{key} gives $\sigma_y \in \underline{\Re}(x(t), \phi_y) $ with $\frac{d}{dt} G(x(t), \sigma_y) < 0$. By definition, $\underline{G}(x(s)) \leq G(x(s),\sigma_y)$ for any $s$, and $\underline{G}(x(t)) = G(x(t), \sigma_y)$. Hence
\begin{align*}
\limsup_{\epsilon \downarrow 0} \frac{\underline{G}(x(t+\epsilon)) - \underline{G}(x(t))}{\epsilon} &\leq 
\limsup_{\epsilon \downarrow 0} \frac{G(x(t+\epsilon), \sigma_y) - G(x(t), \sigma_y)}{\epsilon} \\
&= \frac{d}{dt} G(x(t), \sigma_y) < 0.
\end{align*}
The function $\underline{G} (x(t))$ is continuous and satisfies the above strict inequality for each $t \in [0,1]$, hence it must be strictly decreasing. \\
For $\overline{G} (x(t))$, we similarly use $\sigma_y \in \overline{\Re}(x(t), \phi_y) $ and $\frac{d}{dt} G(x(t), \sigma_y) < 0$ to get
\begin{align*}
\liminf_{\epsilon \downarrow 0} \frac{\overline{G}(x(t - \epsilon)) - \overline{G}(x(t))}{\epsilon} &\geq 
\liminf_{\epsilon \downarrow 0} \frac{G(x(t - \epsilon), \sigma_y) - G(x(t), \sigma_y)}{\epsilon} \\
&= -  \frac{d}{dt} G(x(t), \sigma_y) > 0.
\end{align*}
We again see that $\overline{G} (x(t))$ is strictly decreasing. \\
The case $\alpha >2$ can be proved in a similar fashion, and the proposition follows.
\end{proof}

\section{Optimal Stopping problem for radially symmetric marginals}

Finally, armed with Proposition \ref{consequence}, we establish the main theorem. Recall that in view of Corollary \ref{stay}, we can assume for the minimization problem with $0<\alpha\leq1$ that  $\mu \wedge \nu=0$ without loss of generality.

 \begin{theorem}\label{thm:main}
Suppose $\mu$, $\nu$ are compactly supported, radially symmetric probability measures on $\R^d$, $d\geq2$. Assume $\mu \wedge \nu =0$ and $\mu(\{0\}) =0$. Let $c(x,y) = |x-y|^\alpha$, where $0 < \alpha \neq 2$. Then the solution for the optimization problem \eqref{opt} under the marginals $\mu, \nu$ is unique and it is given by a non-randomized stopping time.
\end{theorem}

\begin{proof}  Fix $0< \alpha<2$, and let $\tau$ be a minimizer for  \eqref{opt}. For $x, y \neq 0$ in $\R^d$ and $x \ne y$, we define the barrier set:
$$U_x^y = \{z \in \R^d : |z| = |y| \text{ and } \langle x, y \rangle < \langle x, z \rangle \}.$$
We shall say that a pair $(f, s)$ and $(g, t)$ in  $S$ is forbidden if  
$$
f(0)=g(0) \neq 0 \text{ and } \exists s' < s \text{ such that } f(s') \in U_{g(0)}^{g(t)}.
$$
In words, a forbidden pair consists of a path that penetrates the barrier generated by the other path.
We denote by ${\bf FP}$ the set of forbidden pairs. 
 First, we claim that there exists  $\Gamma \subset S$ on which $\tau$ is concentrated, such that $\Gamma$ does not admit any ``forbidden pair" that lies in $ \Gamma \times \Gamma$. \\
 Indeed, choose a radially $c$-monotone $\Gamma$ for $\tau$ as in Theorem \ref{rmonoprin} and suppose that ${\rm FP} \cap (\Gamma \times \Gamma) \neq \emptyset$. Then by the Markov property, the penetrating path yields a subharmonic measure, namely the conditional probability, whose barycenter is at the point where the barrier is hit. But this contradicts Corollary \ref{consequence} and Proposition \ref{rmonoprin}. Hence, ${\rm FP} \cap (\Gamma \times \Gamma) = \emptyset$.

Now, we show that  since $\Gamma$ does not allow forbidden pairs, then every stopping time concentrated on $\Gamma$ is necessarily non-randomized, which clearly yield the uniqueness. Indeed, let $\xi$ be any stopping time in ${\cal T}(\mu, \nu)$ with $\xi(\Gamma) = 1$. Note that as $\mu \wedge \nu =0$ and $\mu(\{0\})=0$, we can assume that every stopped path $(f^x, s) \in \Gamma$ has $s>0$ and $x\neq 0$. Now we claim that since ${\rm FP} \cap (\Gamma \times \Gamma) = \emptyset$,  $\xi$ must be of non-randomized type. 

Suppose not. Then there exists an element $(f^x, s) \in \Gamma$ such that the conditional stopping time $\xi^{(f^x, s)}$ is nonzero. This means that the Brownian motion which has followed the path $f^x$ up to time $s>0$ will continue its motion at $y := f^x(s)$. Now consider the barrier $U := U_x^y$ and note that,  Brownian motion starting at $y$ governed by any non-zero stopping time will go through the surface $U$ before its complete stop, since $y$ is on the boundary circle of $U$. This implies that there is a stopped path $(g^y, t) \in S$ such that the concatenation $(f^x\oplus g^y, s+t) \in \Gamma$ and for some $s' < s+t$, $(f^x\oplus g^y) (s') \in U$. This means that the pair $((f^x\oplus g^y, s+t), (f^x, s)) \in \Gamma \times \Gamma$ is a forbidden pair, which is a contradiction.

Therefore, every minimizer is of non-randomized type, and uniqueness follows in the usual way, that is if $\tau$ and $\tau'$ are two minimizers and if  their disintegrations do not agree, i.e. $\tau_{\omega^x} \neq \tau'_{\omega^x}$ for all $\omega^x \in B$ with $\P(B) >0$, then the stopping time  $\frac{{\tau} + {\tau'}}{2}$ must be of randomized type, which is a contradiction since it is obviously a better minimizer.

The maximization problem when $\alpha > 2$, can be proved in a similar fashion. 
\end{proof}

\begin{remark}
 Let $\Gamma$ be the radial $c$-monotone set on which the optimizer in Theorem \ref{thm:main} is concentrated. The proof of Theorem \ref{thm:main} in fact tells us that, for the minimization problem with $0 < \alpha < 2$ or the maximization problem with $\alpha > 2$, the optimal stopping time  is given by the first time Brownian motion $B^x$ hits the following union of barriers
$$\mathcal{U}_x := \cup_y U_x^y, \text{ where } y = f^x(s) \text{ for some } (f^x, s) \in \Gamma. $$
Moreover, by the uniqueness property and the radial symmetry of $\mu$ and $\nu$, the sets $\mathcal{U}_x$'s  are congruent under rotation, that is if $M$ is an orthogonal matrix and $M(x) = x'$, then $M(\mathcal{U}_x) = \mathcal{U}_{x'}$.

For minimization problem with $\alpha > 2$ or maximization problem with $0 < \alpha < 2$, we have the same type of result, but the barrier will be reversed: it will be given by
$$\mathcal{V}_x := \cup_y V_x^y, \text{ where } y = f^x(s) \text{ for some } (f^x, s) \in \Gamma, $$
where $V_x^y$ is the reversed barrier
$$V_x^y = \{z \in \R^d : |z| = |y| \text{ and } \langle x, y \rangle > \langle x, z \rangle \}.$$ 

This is due to the interchange of the superharmonic and subharmonic region of the derivative of the gain function \eqref{sh9}, according to the value of $\alpha$. Also note that when dealing with  maximization problem, the inequalities \eqref{sh5} and \eqref{sh6} in the monotonicity principles must be reversed.

 Finally, we note that the ideas in this paper can be applied to costs that are more general than the ones of the form $|x-y|^\alpha$ considered in this paper. In particular, they apply to a class of cost functions $c(x,y)$ that are invariant under rotation and whose directional derivatives $ \nabla_u c(x,y)$ in the $x$-variable have suitable sub/superharmonic regions in the $y$-variable. 
 
\end{remark}


\bibliographystyle{plain}

\begin{thebibliography}{40}

\bibitem{ap}
L. ~Ambrosio and A. ~Pratelli,
\newblock Existence and stability results in the $L^1$ theory of optimal transportation.
\newblock {\em Optimal transportation and applications. Springer Berlin Heidelberg},
Volume 1813 of the series Lecture Notes in Mathematics (2003) 123--160.



\bibitem{akp}
L. ~Ambrosio, B. ~Kirchheim, and A. ~Pratelli, 
\newblock Existence of optimal transport maps for crystalline norms.
\newblock {\em Duke Mathematical Journal},Volume 125, Number 2 (2004)  207--241.

\bibitem{BaCh74}
J. R.~Baxter and R. V.~Chacon,
\newblock  Potentials of stopped distributions.
\newblock  {\em Illinois J. Math.}, 18:649--656, (1974). MR358960
\bibitem{bch}
M.~Beiglb{\"o}ck, A.M.G.~Cox, and M.~Huesmann,
\newblock Optimal transport and Skorokhod embedding.
\newblock {\em Inventiones mathematicae}, May 2017, Volume 208, Issue 2, pp 327--400.

\bibitem{BeHePe11}
M.~Beiglb{\"o}ck, P.~{\HL}, and F.~Penkner, 
\newblock Model-independent bounds for option prices--a mass transport approach.
\newblock {\em Finance and Stochastics}, 17 (2013) 477--501.

\bibitem{BeHeTo}
M.~Beiglb{\"o}ck, P.~{\HL}, and N.~Touzi, 
\newblock Monotone martingale transport plans and Skorohod Embedding.
\newblock {\em Stochastic Processes and their Applications}, Volume 127, Issue 9, September 2017, 3005--3013.

\bibitem{BeJu16}
M.~Beiglb{\"o}ck and N.~Juillet, 
\newblock On a problem of optimal transport under marginal martingale constraints.
\newblock {\em Ann. Probab}, Volume 44, Number 1 (2016), 42--106.

\bibitem{BeNuTo17}
M.~Beiglb{\"o}ck, M.~Nutz, and N.~Touzi,
\newblock Complete duality for martingale optimal transport on the line.
\newblock  {\em Ann. Probab.}, Volume 45, Number 5 (2017), 3038--3074.



\bibitem{br}
Y.~Brenier,
\newblock Polar factorization and monotone rearrangement of vector-valued functions.
\newblock {\em Communications on pure and applied mathematics}, Volume 44, Issue 4 (1991) 375--417.

\bibitem{BuSc92}
S.~Bu and W.~Schachermayer,
\newblock Approximation of Jensen measures by image measures under holomorphic functions and applications.
\newblock {\em Trans. Amer. Math. Soc.}, 331 (1992), 585--608.

\bibitem{ChWa76}
R. V.~Chacon and J. B.~Walsh,
\newblock One-dimensional potential embedding. 
\newblock In {\em S{\'e}minaire de Probabilit{\'e}s,  X.  Lecture Notes in Math.}, Vol. 511. Springer, Berlin (1976),  19--23. MR445598.

\bibitem{Ch59}
G.~Choquet,
\newblock Forme abstraite du th{\'e}or{\`e}me de capacitabilit{\'e}.
\newblock {\em Ann. Inst. Fourier. Grenoble}, 9 (1959) 83--89. 

\bibitem{COT}
A.M.G.~Cox, J. \OB, and N.~Touzi,
\newblock The Root solution to the multi-marginal embedding problem: an optimal stopping and time-reversal approach.
\newblock http://arxiv.org/abs/1505.03169 (2015).

\bibitem{DeTo17}
H.~De March and N.~Touzi,
\newblock Irreducible convex paving for decomposition of multi-dimensional martingale transport plans.
\newblock https://arxiv.org/abs/1702.08298 (2017).

\bibitem{ds1}
Y.~Dolinsky and H.M.~Soner, 
\newblock Martingale optimal transport and robust hedging in continuous time.
\newblock {\em Probab. Theory Relat. Fields}, 160 (2014) 391--427.

\bibitem{ds2}
Y.~Dolinsky and H.M.~Soner, 
\newblock Martingale optimal transport in the Skorokhod space.
\newblock {\em Stochastic Processes and their Applications}, 125(10) (2015) 3893--3931.

\bibitem{Ed85}
G.A.~Edgar, 
\newblock Complex martingale convergence.
\newblock {\em Banach Spaces,
Volume 1166 of the series Lecture Notes in Mathematics}, 38--59.

\bibitem{Ed86}
G.A.~Edgar, 
\newblock Analytic martingale convergence.
\newblock {\em Journal of Functional Analysis},
Volume 69, Issue 2, November 1986, 268--280.

\bibitem{Fa80}
N.~Falkner,
\newblock  On Skorohod embedding in n-dimensional Brownian motion by means of natural stopping times. 
\newblock  In {\em S{\'e}minaire de Probabilit{\'e}s, XIV, volume 784 of Lecture Notes in Math.}, 357--391. Springer, Berlin, (1980). MR580142.


\bibitem{GaHeTo11}
A.~Galichon, P.~Henry-Labordere, and N.~Touzi, 
\newblock A Stochastic Control Approach to No-Arbitrage Bounds Given
  Marginals, with an Application to Lookback Options.
\newblock {\em Ann. Appl. Probab}, Volume 24, Number 1 (2014) 312--336.

\bibitem{gm}
W.~Gangbo and R.J.~McCann, 
\newblock The geometry of optimal transportation.
\newblock {\em Acta Math}, Volume 177, Issue 2 (1996) 113--161. 


\bibitem{GKL2}
N.~Ghoussoub, Y-H. ~Kim, and T. ~Lim, 
\newblock 
Structure of optimal martingale transport in general dimensions.
\newblock {\em http://arxiv.org/abs/1508.01806,} (2015).

\bibitem{GKL3}
N.~Ghoussoub, Y-H. ~Kim, and T. ~Lim, 
\newblock 
Structure of optimal subharmonic martingale transport in general dimensions.
\newblock In preparation.

\bibitem{GhMa89}
N.~Ghoussoub and B.~Maurey, 
\newblock Plurisubharmonic martingales and barriers in complex quasi-Banach spaces.
\newblock {\em Annales de l'institut Fourier}, 
Volume 39, Issue 4 (1989) 1007--1060.

\bibitem{GTT}
G.~Guo, X.~Tan, and N.~Touzi, 
\newblock 
Optimal Skorokhod embedding under finitely-many marginal constraints.
\newblock {\em SIAM J. Control Optim}, 54(4), 2174--2201.



\bibitem{Ho98}
D.~Hobson,
\newblock Robust hedging of the lookback option, 
\newblock {\em Finance and Stochastics},  2 (1998) 329--347.

\bibitem{Ho11}
D.~Hobson, 
\newblock The Skorokhod embedding problem and model-independent bounds for
  option prices.
\newblock In {\em Paris-{P}rinceton {L}ectures on {M}athematical {F}inance
  2010}, volume 2003 of {\em Lecture Notes in Math}, Springer,
  Berlin (2011)  267--318.

\bibitem{HoKl12}
D.~Hobson and M.~Klimmek, 
\newblock , Robust price bounds for the forward starting straddle.
\newblock {\em Finance and Stochastics}, Volume 19, Issue 1 (2014) 189--214.


\bibitem{HoNe11}
D.~Hobson and A.~Neuberger, 
\newblock Robust bounds for forward start options.
\newblock {\em Mathematical Finance}, Volume 22, Issue 1 (2012) 31--56.

\bibitem{Johnson}
J. A.~Johnson, 
\newblock Banach spaces of Lipschitz functions and vector-valued Lipschitz functions.
\newblock {\em Bull. Amer. Math. Soc}, Volume 75, Number 6 (1969), 1334--1338.


\bibitem{KTT}
S.~K{\"a}llblad, X.~Tan and N.~Touzi,
\newblock Optimal Skorokhod embedding given full marginals and Azema-Yor peacocks.
\newblock {\em Ann. Appl. Probab}, Volume 27, Number 2 (2017), 686--719.


\bibitem{Lim}
T.~Lim,
\newblock Optimal martingale transport between radially symmetric
marginals in general dimensions.
\newblock https://arxiv.org/abs/1412.3530, (2016).

\bibitem{Mo72}
I.~Monroe,
\newblock On embedding right continuous martingales in Brownian motion.
\newblock {\em The Annals of Mathematical Statistics}, Vol. 43, No. 4, 1293--1311 (1972).

\bibitem{Ob04}
J.~Ob{\l}{\'o}j,
\newblock  The Skorokhod embedding problem and its offspring
\newblock {\em  Probability Surveys}, 1 (2004) 321--392.

\bibitem{ObSi17}
J.~Ob{\l}{\'o}j and P.~Siorpaes,
\newblock Structure of martingale transports in finite dimensions.
\newblock https://arxiv.org/abs/1702.08433 (2017).

\bibitem{ObSp16}
J.~Ob{\l}{\'o}j and P.~Spoida,
\newblock An Iterated Az\'{e} ma-Yor Type Embedding for Finitely Many Marginals.
\newblock {\em Ann. Probab.}, to appear. (2016)

\bibitem{Rock}
R.T.~Rockafellar,
\newblock  Characterization of the subdifferentials of convex functions.
 \newblock {\em Pacific J. Math}, 17 (1966) 497--510. 
 
 \bibitem{Ro71}
 H.~Rost,
\newblock The Stopping Distributions of a Markov Process.
\newblock {\em Inventiones mathematicae}, 14, 1--16 (1971).



\bibitem{Sk65}
A.V.~Skorokhod,
\newblock  Studies in the Theory of Random Processes (Translated from the Russian
 by  Scripta Technica Inc.). 
 \newblock {\em Addison-Wesley Publishing Co. Inc.}, Reading (1965)

\bibitem{St65}
V.~Strassen,
\newblock The existence of probability measures with given marginals.
\newblock {\em Ann. Math. Statist.}, 36 (1965) 423--439. 

\bibitem{Vi03}
C.~Villani,
\newblock {\em Topics in optimal transportation}, Volume~58, {\em Graduate
  Studies in Mathematics}.
\newblock American Mathematical Society, Providence, RI, 2003.

\bibitem{Vi09}
C.~Villani,
\newblock {\em Optimal Transport. Old and New}, Vol. 338, {\em Grundlehren
  der mathematischen Wissenschaften}.
\newblock Springer, 2009.

\end{thebibliography}


\end{document}